\documentclass[11pt,leqno]{article}

\usepackage[all]{xy}
\usepackage{amssymb,amsmath,amsthm,   url,rotating}
\usepackage{mathrsfs}
\usepackage{hyperref}
  \usepackage{graphicx,epsfig}
  \usepackage{xcolor,graphicx}

    \usepackage{makeidx}
\newcommand{\n}{\noindent}

\newcommand{\vp}{\varepsilon}
\newcommand{\bb}[1]{\mathbb{#1}}
\newcommand{\cl}[1]{\mathcal{#1}}

\theoremstyle{plain}
\newtheorem{thm}{Theorem}[section]
\newtheorem{lem}[thm]{Lemma}

\newtheorem{pro}[thm]{Proposition}

\newtheorem{cor}[thm]{Corollary}

\theoremstyle{definition}

\newtheorem{dfn}[thm]{Definition}

\theoremstyle{remark}
\newtheorem{rem}[thm]{Remark}

\numberwithin{equation}{section}

\def\tilde{\widetilde}

\renewcommand{\tilde}{\widetilde}

 \def\R{\bb  R}
\def\RR{\bb  R}
 
\def\CC{\bb  C}
\def\KK{\bb  K}
\def\E{\bb  E}
\def\F{\bb  F}
\def\P{\bb  P}
\def\T{\bb  T}

\def\d{\delta}
\def\NN{\bb  N}
\def\N{\bb  N}
\def\RR{\bb  R}
\def\Q{\bb  Q}
\def\PP{\bb  P}
\def\CC{\bb  C}
\def\KK{\bb  K}
\def\RR{\bb  R}
\def\CC{\bb  C}
\def\KK{\bb  K}
\def\E{\bb  E}
\def\F{\bb  F}
\def\P{\bb  P}
\def\T{\bb  T}

\def\d{\delta}
\def\NN{\bb  N}
\def\ZZ{\bb  Z}
\def\Z{\bb  Z}
\def\RR{\bb  R}
\def\PP{\bb  P}
\def\CC{\bb  C}
\def\KK{\bb  K}
\def\FF{\bb  F}

\def\ov{\overline}
\def\phi{\varphi}
\def\ie{{\it  i.e.\  }}
\def\eg{{\it  e.g.}}
\def\n{\noindent}
\def\nl{\nolimits}
\def\tr{\rm  tr}
\def\a{\alpha}

\def\C{\mathscr{C}}
\def\B{\mathscr{B}}
\def\I{\cl  I}
\def\e{\cl  E}

\begin{document}

\title{ The lifting property for $C^*$-algebras : from local to global ?
    }

\author{by\\  
Gilles  Pisier  \\
Texas  A\&M  University\\
College  Station,  TX  77843,  U.  S.  A.}

\def\C{\mathscr{C}}
\def\B{\mathscr{B}}
\def\I{\cl  I}
\def\e{\cl  E}
\def\G{\bb G}

\def\R{\bb R}
\def\RR{\bb R}
 
\def\CC{\bb C}
\def\KK{\bb K}
\def\E{\bb E}
\def\F{\bb F}
\def\P{\bb P}
\def\T{\bb T}

\def\d{\delta}
\def\NN{\bb N}
\def\N{\bb N}
\def\RR{\bb R}
\def\Q{\bb Q}
\def\PP{\bb P}
\def\CC{\bb C}
\def\KK{\bb K}
\def\RR{\bb R}
\def\CC{\bb C}
\def\KK{\bb K}
\def\E{\bb E}
\def\F{\bb F}
\def\P{\bb P}
\def\T{\bb T}

\def\d{\delta}
\def\NN{\bb N}
\def\ZZ{\bb Z}
\def\Z{\bb Z}
\def\RR{\bb R}
\def\PP{\bb P}
\def\CC{\bb C}
\def\KK{\bb K}
\def\FF{\bb F}

\def\v{\varphi}
\def\ov{\overline}
\def\phi{\varphi}
\def\ie{{\it i.e.\ }}
\def\eg{{\it e.g.}}
\def\n{\noindent}
\def\nl{\nolimits}
\def\tr{{\rm tr}}

 \pagenumbering{roman}
 \maketitle
 
\begin{abstract}    This note is motivated by Kirchberg's conjecture
that the local lifting property (LLP) implies
the lifting property (LP) for separable $C^*$-algebras.
The author recently constructed by a ``local" method an  example of 
$C^*$-algebra with the LLP  and the weak expectation property
(WEP) which might be a counterexample.
 We give here several  conditions
 all equivalent to 
  the validity
 of the implication LLP $\Rightarrow$ LP
 for WEP $C^*$-algebras, or equivalently for quotients of WEP
 $C^*$-algebras (QWEP),  that hopefully clarify the 
 nature of the problem.
 But unfortunately
 we cannot decide whether they  
 hold true. These conditions highlight the notion of ``controllable" finite dimensional (f.d.) operator space, in connection with certain $C^*$-tensor products. The latter led us to 
 a closely related variant of  local reflexivity.
 \end{abstract}

 \thispagestyle{empty}

\setcounter{page}{1}

 \pagenumbering{arabic}

 

\bigskip

 Let $A,C$ be $C^*$-algebras. Let $I\subset C$ be a closed self-adjoint two-sided ideal so that the quotient $C/I$ is a $C^*$-algebra.
 Assuming $A$ unital, let $S \subset  A$ be an operator system, meaning a self-adjoint subspace
 containing the unit.
 Consider a map $u: S \to C/I$.
 We say that $u$ is liftable if there is a unital completely positive 
 (u.c.p. in short) map 
$ \hat u : S \to C$ lifting $u$. We say that $u$ is locally liftable
if for any finite dimensional (f.d. in short) operator system $E\subset A$
the restriction $u_{|E}$ is liftable.

 In \cite{Kir} Kirchberg defined the lifting property (LP in short) and the local lifting property (LLP in short) for a unital $C^*$-algebra $A$ as follows:\\
 The algebra $A$ has the LP (resp. LLP) if any u.c.p.
 map $u: A \to C/I$ into an arbitrary quotient is liftable (resp. locally liftable).\\
 When $A$ is not unital,    $A$ has the LP (resp. LLP) if its unitization does.
 
  In some sense $C^*(\F_\infty )$ is the fundamental example of LP
  for $C^*$-algebras. It plays an analogous role to
  that of $\ell_1$-spaces in Banach space theory.
  
 In the remarkable last section of his outstanding paper  \cite{Kir}, Kirchberg
  formulated seven equivalent conjectures (labelled as B1,$\cdots$, B7) all equivalent
 to the validity of the Connes embedding problem (labelled as B1)
 which asks whether any finite von Neumann algebra
 on $\ell_2$ embeds in an ultraproduct of matrix algebras.  
 One of these conjectures (namely B3) states that there is a unique $C^*$-norm on $C^*(\F_\infty ) \otimes C^*(\F_\infty )$ where $\F_\infty$ denotes the free group
 with  infinitely countably many generators.
 In the   paper
\cite{JNVWY} (see also \cite{Vid}) a negative solution is proposed, based on the equivalence
of B3 with an equivalent conjecture of Tsirelson developed in quantum information theory. See \cite{P6} for a detailed account of the equivalence of these various conjectures. 

In \cite{Kir} Kirchberg proved that his (equivalent) B-conjectures
 imply his conjecture D which states that for separable $C^*$-algebras
 the LLP implies the LP (in other words they are equivalent).
 The latter conjecture is the motivation for the present paper.
 Although we do not reach a conclusion, we hope that our results
 may advance the situation. 
 The difficulty is that,
 apart from the nuclear $C^*$-algebras, very
 few examples of LLP $C^*$-algebras are known except for $C^*(\F_\infty )$
  and  algebras derived from it. 
However, we recently constructed
 in \cite{155} a new type of non-nuclear LLP $C^*$-algebras, 
 for which it does not seem clear at all that they have the LP;
the failure of all our attempts to either prove or disprove that is what led to the present paper. 
 The special feature of the latter examples is that they have Lance's weak expectation property (WEP in short), which by another of  Kirchberg's results in \cite{Kir} is equivalent to the statement that, if $A$ denotes the $C^*$-algebra,
 there is a unique $C^*$-norm on $A\otimes C^*(\F_\infty )$.
 Kirchberg's   conjecture B3 states that $C^*(\F_\infty )$ has the WEP or equivalently that {\it } any separable  $C^*$-algebra is a quotient of a WEP 
 $C^*$-algebra (these are called QWEP).
 He also proved in \cite{Kir} that $A$ has the LLP if and only if
  there is a unique $C^*$-norm on $A\otimes B(\ell_2)$. Thus the question whether   the algebras we constructed in \cite{155} have the LP
  seemed like a very natural way to tackle Kirchberg's conjecture D.
  Two roads opened to us: \\
  $\bullet$ The most ambitious one would be
  to prove that one of these algebras  (which have LLP and WEP but are not nuclear) fails the LP. This would invalidate conjecture D and hence give an alternate and hopefully much simpler solution to the Connes embedding problem.
 \\
   $\bullet$ A much less ambitious project would be   to prove that
   conjecture D actually  implies (conversely) the (equivalent) conjectures B. In that case
   one would deduce a negative answer for D from a negative
   solution to the Connes embedding problem.
 
 This leaves aside the possibility that conjecture D might hold at least for all
WEP (or all QWEP) $C^*$-algebras in which case our examples in \cite{155}
would have the LP.  
However, if it turns out that
 conjecture D is not valid, exhibiting a WEP $C^*$-algebra with the LP remains
  a quite interesting goal,  considerably strengthening   \cite{155}.

Unfortunately, we have been unable to decide any of these three alternatives.
Nevertheless, we obtained several equivalent forms
of the validity of LLP $\Rightarrow$ LP for WEP (and separable)
 $C^*$-algebras, which seem very close to the goal of 
 at least the less ambitious
 project.
 
 For instance, we show that if a single one of our examples in \cite{155}
 has the LP then all WEP and LLP (separable) $C^*$-algebras also have LP.
 Moreover,  denoting for simplicity
 $$ \C=C^*(\F_\infty)$$
 we prove the equivalence of the following (conjectural) assertions:\\
  {\rm (i)} For WEP (and separable)
 $C^*$-algebras, the LLP implies the LP.\\
{\rm (ii)} There is a completely isometric embedding $f: \C \to \C$ such that the minimal and maximal $C^*$-norms
   coincide on  the { algebraic} tensor product 
   $ f(\C ) \otimes \C $.
\\{\rm (iii)} Any (or some) completely  isometric embedding
$j: \C \to B(\ell_2) $ is such that
the map
 $$Id_{\ell_\infty(\C)} \otimes j   :    \ell_\infty(\C)\otimes     \C \to \ell_\infty(\C) \otimes   B(\ell_2)$$
 is contractive from the minimal tensor product to the maximal one.

We feel (ii) is very close to Kirchberg's B3 conjecture, which asserts that
the min and max norms are equivalent on $\C \otimes \C$. Indeed, the latter is just
(ii) for a surjective $f$.

On the discouraging side, even if there is a   a $*$-homomorphism $f$ 
as in (ii) 
we do not know
   whether   B3 follows.

Our work also leads to some surprising classes of f.d. operator subspaces
of $\C$.
Let us say that a  f.d.  subspace $E\subset \C$ is {\it rigid}
if any completely contractive map $u: E \to \C$ is the restriction to $E$ of a $*$-homomorphism. If $E\subset \C$  the dual operator space
$E^*$ also embeds completely isometrically in $\C$ (and conversely). The spaces 
 that are completely isometric to a space with rigid  dual
are called  {\it controllable}. 
We show that   (i) is equivalent to the assertion that {\it all} 
 f.d.  subspaces $E\subset \C$ are completely isometric to rigid ones (see Lemma \ref{LP4} and Theorem \ref{t1}). Equivalently all subspaces $E\subset \C$ are controllable.
 The controllable  spaces form an interesting class of f.d. operator subspaces of $
 \C$ that probably deserves more investigation.
 It can be shown that if the dual $E^*$ is 1-exact then $E$ is completely isometric to a rigid space but
 at the time of this writing we could not decide whether the latter is true  when $E$ is the
 $n$-dimensional  commutative $C^*$-algebra $\ell_\infty^n$.
  We conjecture it is not,   already for $n=3$ i.e. for $\ell_\infty^3$. Equivalently,
  we conjecture that its operator space dual $\ell_1^3$    is not controllable.
  Actually, the conjecture that $M_k(\ell_1^3)$ is controllable for all $k\ge 1$ is equivalent
  to the validity of the assertion (i) by the linearization
  trick of \cite{93.} (see Remark \ref{r2}).

We will always restrict our considerations of the LP to the separable case.
In the non-separable case very little seems to be known on the LP.
It is apparently open whether $C^*(\F)$ has the LP when $\F$ is an
uncountable free group. In the non-separable case,
it is natural to consider (as e.g. in \cite[\S  7]{157}) the class of $C^*$-algebras $A$
such that all separable $C^*$-subalgebras $B\subset A$
are contained in a larger separable one with LP.

We end the paper with a generalization of the local reflexivity (LR in short) property
on which our previous paper \cite{157} is based.
    \medskip
    
  \n{\bf Notation and   background}
    
      \medskip
    
   \n As often, we abbreviate completely bounded by c.b.,
    completely positive by c.p. and completely contractive by c.c.
    We will also abbreviate operator space (or operator spaces) by o.s.
    and finite dimensional by f.d.
    We abbreviate completely isometric
    and completely isometrically by c.i. \\
    An o.s. is a closed subspace of a $C^*$-algebra or of $B(H)$.
    The duality of o.s. is a consequence of Ruan's characterization
    (see e.g.  \cite{ER}) of the sequences of norms on $(M_n(E))_{n\ge 1}$ (where $M_n(E)$ denotes the space of $n \times n$-matrices with entries in a vector space $E$) that come from 
    an embedding of $E$ into $B(H)$ for some $H$.
Given an o.s. $E \subset B(H)$ with (Banach space sense) dual $E^*$  there is an $\cl H$ and an isometric embedding
$j: E^*\to B(\cl H)$  that induces 
isometric isomorphisms $M_n(j(E^*) )\simeq CB(E,M_n)$ for all $n$.
The embedding $j$ allows one to consider $E^*$ as an o.s.
This is what is called the dual o.s. structure on $E^*$.
We will often refer to it as the o.s. dual of $E$. An important property due independently to
Effros-Ruan and Blecher-Paulsen (see e.g. \cite[p. 42]{ER} or \cite[p. 41]{P4})
is that $E\simeq (E^*)^*$ c.i. for any f.d. $E$.

We denote by $Id_E$ the identity map on $E$.
    
    We reserve the notation $E\otimes F$ for the {\it algebraic} tensor product of two o.s.
    
    Recall that if $E,F$ are completely isomorphic operator spaces, we set
$$d_{cb}(E,F)=\inf \{ \|u\|_{cb} \|u^{-1}\|_{cb} \}$$
where the inf runs over all the complete isomorphisms $u:\ E \to F$.

 \begin{dfn}\label{d10} We will say that an o.s. $X$ locally embeds in an o.s. $Y$ if
  for any f.d. $E\subset X$ and any $\vp>0$ there is $F\subset Y$ such that
  $d_{cb}(E,F) <1+\vp$.
\end{dfn}

We refer to \cite{ER,P4} for more background on operator spaces.

Let   $(D_i)_{i\in I}$ be a   family 
 of $C^*$-algebras. We will denote by $\ell_\infty(\{D_i\mid \in I\})$ or simply by
 $\ell_\infty(\{D_i\})$ the $C^*$-algebra formed of   the   families
    $d=(d_i)_{i\in I}$ in $\prod_{i\in I} D_i$ such that $\sup\nl_{i\in I} \|d_i\|<\infty$, equipped
    with the norm $d\mapsto \sup\nl_{i\in I}\|d_i\|$.

 The following classical fact will be invoked repeatedly.
 \begin{lem}\label{nsf} Let $\F$ be any uncountable free group and let $X\subset C^*(\F)$
 be a separable subspace. There is  a $C^*$-subalgebra $\C'\subset C^*(\F)$
  such that $X \subset \C'$,  the algebra
  $\C'$ is (unitally) isomorphic to $C^*(\F_\infty)$ and there is a unital
  c.p. and contractive projection $P:  C^*(\F)\to \C'$. Consequently,
  for any other $C^*$-algebra $A$ the norm induced on $A\otimes \C'$
  (algebraic tensor product)  
  by $A\otimes_{\max} C^*(\F)$ coincides with the norm in $A\otimes_{\max} \C'$.
 \end{lem}
 \begin{proof} Since $X$ is separable it lies in the closure  $\C'$ of 
 a subgroup of $\F$ generated by a countably infinite set of free generators.
 The latter is a copy of $\F_\infty$. As is well known (see e.g. \cite{P6}) there is a conditional expectation operator   $P:  C^*(\F)\to \C'$.
 \end{proof}
 
 \begin{rem}[On Arveson's principle]
 We will often invoke Arveson's principle. This
  deals with certain  classes   $\cl F(X,C)$ (that we call admissible) formed of maps from a separable
o.s. $X$ to a $C^*$-algebra  $C$. Let $q: C \to Q$ be a surjective $*$-homomorphism
between $C^*$-algebras. A map 
$u: A \to Q$ is called $\cl F$-liftable if there is $\hat u\in \cl F(X,C)$ such that
$u=q \hat u$.
 Arveson's principle  (see \cite[p. 351]{[Ar4]})  says that pointwise limits of $\cl F$-liftable maps are $\cl F$-liftable.
 The typical example  of admissible $\cl F$ is $\cl F(X,C)=B_{CB(X,C)}$.
 Assuming $X$ is a $C^*$-algebra, then $\cl F(X,C)=\{ u \text{ c.c. and c.p. }\}$ is another basic example.
 By definition a class is   admissible if $\cl F(X,C)$ is bounded in norm,  pointwise closed
 and satisfies the following form of non-commutative convexity:
   for any pair $f,g$ in $\cl F$
   and any $\sigma \in C_+$ with $\|\sigma\| \le 1$ 
   the mapping
   $$x\mapsto \sigma^{1/2} f(x)  \sigma^{1/2}+ (1-\sigma)^{1/2} g(x) (1-\sigma)^{1/2} $$
   belongs to $\cl F$. 
   \end{rem}
 
  \begin{rem}[Examples]\label{r10}
  Choi and Effros \cite{[CE5]} proved that all separable nuclear $C^*$-algebras
  have the LP. Later on, Kirchberg \cite{Kiuhf}  proved that $C^*(\F)$
  has the LP when $\F$ is any countable free group. While the basic idea
  to lift homomorphisms from $\F$ to the unitary group of
  a quotient $A/I$   is somewhat  intuitive,
  the  proof for general u.c.p. maps is non trivial. 
  It uses Kasparov's version of the Stinespring dilation for such maps. See \cite{BO}
  for an account of that first proof. A different second proof  
  can be derived from
  \cite{157}, using the criterion appearing in Theorem \ref{t15},
    the description of the unit ball of $E\otimes_{\max} C$
  when $E\subset \C$ is the span of 
  a finite set of free unitary generators (see (7) and (9) in  \cite{93.}), and the linearization trick of
  \cite{93.}. 
  That second proof can be equivalently rewritten  by combining the proof of  \cite[Th. 17]{93.}
  with the criterion involving the nor-tensor product in \cite[Th. 0.13]{157}.
  These two proofs seem to be the only known ones that $C^*(\F)$
  has the LP. 
  \end{rem}
    
  \begin{rem}[Stability under free product]
  In   \cite{93.} we proved that the free product of 
  an arbitrary family of unital $C^*$-algebras $(A_i)_{i\in I}$ with LLP
  also has the LLP. Here and below by ``free product" we mean
  the full (or maximal) unital free product. \\
  Let $A$ be a separable unital $C^*$-algebra and let $q: C^*(\F) \to A$ be a surjective unital $*$-homomorphism. Then
$A$ has the LP if and only if there is a u.c.p. map
$u: A \to C^*(\F)$ lifting $q$. Indeed, if this holds the LP 
will be inherited by $A$ from $C^*(\F)$. 
 Taking the LP of $C^*(\F)$ for granted,  Boca proved in \cite{Boca2} that the LP is stable by full countable free products of unital $C^*$-algebra $(A_i)_{i\in I}$. He uses a simple
  application of his remarkable result on free products of u.c.p. maps
  in \cite{Boca} (see also \cite{Boca3,DK}).
  This can be applied
  to an arbitrary family of unital $C^*$-algebras $(A_i)_{i\in I}$ and
  arbitrary free groups
  with surjective unital $*$-homomorphisms $q_i: C^*(\F_i) \to A_i$ admitting
u.c.p. liftings $u_i: A_i \to C^*(\F_i)$.
Boca's theorem from \cite{Boca} yields a u.c.p. lifting
for the quotient map $\ast_{i\in I} q_i : C^*(\ast_{i\in I} \F_i) \to \ast_{i\in I} A_i$.
In the separable case this shows that the LP is stable by countable free products,
but one first needs to know that $C^*(\F_\infty)$ has the LP
(for which two proofs are described in Remark \ref{r10}).
  \\
  A direct proof (from scratch) can  be derived from the more recent criterion in Theorem \ref{t15}
  and  the linearization trick of 
  \cite{93.}. This is similar to the second proof in Remark \ref{r10}.
  
Given this,  the best way to unify the Choi-Effros and Kirchberg theorems
is perhaps to state that any  
 countable free product of separable unital {\it nuclear} $C^*$-algebras has the LP.

   \end{rem}
 \begin{rem}[Counterexamples] It is natural to say that a discrete group $G$ has the LP
 if $C^*(G)$ has it. 
 In the present remark all groups are assumed {\it countable}.
 By the preceding discussion we know that   free groups and   amenable groups have the LP. Thus it is not  at all obvious to find groups failing the LP or the LLP. As it turns out all the known counterexamples to the LP or the LLP  use Kazhdan's property (T). Ozawa \cite{Ozpams} proved that a continuum  of (T) groups failing the LP 
 exists but without producing an explicit example. Later on    Thom 
  \cite{Thom} produced an explicit example failing the LLP.
  More recently, Ioana, Spaas and Wiersma
 showed that $SL(n,\Z)$ fails the LLP for all $n\ge 3$.
 They also showed that 
  any   infinitely presented group
with property (T) must fail the LP. By \cite{Ozpams}  there is a continuum of such groups.
 It remains an open question whether all discrete
 groups with property (T) fail the LP or even the LLP.
    \end{rem}

 \section{Preliminaries}
 
 We start by recalling the main result from \cite{157} on the LP.

 \begin{thm}\label{t15} Let $\C=C^*(\F_\infty)$.
 A separable $C^*$-algebra $A$ has the LP if and only if
 for any set $I$, any f.d. $E\subset A$ and any 
 family $t=(t_i)_{i\in I}$ in $\C \otimes E$ we have
 \begin{equation}\label{e15}
 \|t\|_{\ell_\infty(I,\C) \otimes_{\max}  A} \le \sup_{i\in I}
 \|t_i\|_{\C \otimes_{\max}  A} .\end{equation}
 For this to hold it suffices to check the case when $I=\N$.\\
 Moreover, \eqref{e15} holds whenever $A=C^*(\F)$ for any
 (not necessarly countable) free group.
 \end{thm}
 
 \begin{rem}   Assume that the Kirchberg conjecture (and the Connes embedding problem) holds. 
 We will show  that LLP $\Rightarrow$ LP for separable $C^*$-algebras
 (which is the same as conjecture (B3) $\Rightarrow$ (D) in  Kirchberg's  \cite{Kir}).
  The conjecture is that     $\C$ has the WEP.
 By a general result $\ell_\infty(\C)$ also has it (see \cite[p. 193]{P6}).
    By Kirchberg's fundamental theorem (see \cite[p. 195]{P6}) we have
    $B\otimes_{\min} C=B\otimes_{\max} C$ whenever $C$ has LLP and $B$ has WEP. Therefore
     $\ell_\infty(\C) \otimes_{\min} C=\ell_\infty(\C) \otimes_{\max} C$.
     Using this last identity, Theorem \ref{t15} shows that $C$ has the LP  if it is separable and has the LLP.
 \end{rem}
  \begin{rem} Note that $\ell_\infty(\C)$ fails the LLP.
  Indeed, by \cite{JP}
  the algebra $\ell_\infty(\{M_n\mid n\ge 1\})$ 
  (which clearly embeds c.i. in $\ell_\infty(\C)$)  does not locally embed in $\C$.
More generally any $C^*$-algebra $C$ that contains c.i. copies of $M_n$ for all $n$ is such that $\ell_\infty(C)$ fails LLP.
  \end{rem}
  In companion to  $\C=C^*(\F_\infty)$, let $\B=B(\ell_2)$.
  We denote by $$j : \C \to \B$$
  an isometric $*$-homomorphism.
  
  Let $A,B,C$ be $C^*$-algebras.
  Let $E\subset A$. We say that a mapping $u: E \to B$
  is $C$-nuclear if
  $$\| id_C \otimes u: C \otimes_{\min} E \to C\otimes_{\max}  B \|=1
 ,$$
  or equivalently if $  \| u \otimes id_C: E \otimes_{\min} C \to B\otimes_{\max}  C \|=1.$
  
  With this terminology, a $C^*$-algebra $A$ is WEP (resp. LLP) if and only if $Id_A$ is $\C$-nuclear
  (resp. $\B$-nuclear). See \cite{P6} for an exposition of these topics.

  \begin{lem}\label{l1} 
  Assume that $B$ satisfies \eqref{e15}  (e.g.  $B=\C$ or $B=C^*(\F)$ for any
   free group $\F$). Then any $\C$-nuclear 
   mapping $u: E \to B$ is  $\ell_\infty(\C)$-nuclear.
 \end{lem}
   \begin{proof}
 Consider 
  $$Id \otimes u: \ell_\infty(\C) \otimes E \to \ell_\infty(\C)  \otimes  B.$$
  Note that we have an isometric inclusion
  $$\ell_\infty(\C)   \otimes_{\min} E \subset \ell_\infty(\C  \otimes_{\min} E) $$
  while  by \eqref{e15} we have    an isometric inclusion
   $$\ell_\infty(\C)  \otimes_{\max} B \subset \ell_\infty(\C  \otimes_{\max} B) .$$
  If $u$ is $\C$-nuclear,  obviously $Id_\C   \otimes  u$ defines a  contractive map
   $$\ell_\infty(\C  \otimes_{\min} E) \to \ell_\infty(\C  \otimes_{\max} B),$$
  from which the lemma follows by restriction to $\ell_\infty(\C)   \otimes E$.
  \end{proof}
  
 Let $E\subset A$ and $u: E \to B$. We will denote the norm that controls
 $\ell_\infty(\C)$-nuclearity by $\|u\|_{mM}$, i.e. we set
 \begin{equation}\label{mM}
  \|u\|_{mM}=\| Id_{\ell_\infty(\C)} \otimes u: \ell_\infty(\C) \otimes_{\min} E \to 
  \ell_\infty(\C) \otimes_{\max} B\|.\end{equation}
  
  Thus Lemma \ref{l1} tells us that if $B$  satisfies \eqref{e15}
 \begin{equation}\label{mM'}
  \|u\|_{mM}=\| Id_{\C}  \otimes u: \C \otimes_{\min} E \to 
  \C \otimes_{\max} B\|.   \end{equation}
  
 If moreover $A$ has  the WEP (with  $B$ still satisfying \eqref{e15}) then for any $u: A \to B$ we have
 \begin{equation}\label{mM11}
  \|u\|_{mM}=\| Id_{\C}  \otimes u: \C \otimes_{\max} A \to 
  \C \otimes_{\max} B\|.   \end{equation}
  
  \begin{pro}\label{p1} A separable WEP $C^*$-algebra $A$ has the LP
  if and only if $Id_A$ is $\ell_\infty(\C)$-nuclear, or equivalently if and only if
  $$  \|Id_A\|_{mM}=1.$$
  \end{pro}
   \begin{proof} Assume $  \|Id_A\|_{mM}=1.$ A fortiori the pair $(A, \C)$
   is nuclear so $A$ is WEP. In addition 
   we have isometrically
   $$\ell_\infty(\C) \otimes_{\max} A=
   \ell_\infty(\C) \otimes_{\min} A
   \subset \ell_\infty(\C \otimes_{\min} A) =\ell_\infty(\C \otimes_{\max} A),$$
   which means $A$ has the LP by Theorem \ref{t15}.\\
   Conversely, if $A$ has both LP and WEP, then $(A, \C)$ is a nuclear pair
   and 
   $$\ell_\infty(\C) \otimes_{\max} A\subset \ell_\infty(\C \otimes_{\max} A) =\ell_\infty(\C \otimes_{\min} A)$$
   and we also have isometrically
  $$ \ell_\infty(\C) \otimes_{\min} A
   \subset \ell_\infty(\C \otimes_{\min} A) ,$$
  so we conclude that $\ell_\infty(\C) \otimes_{\max} A=\ell_\infty(\C) \otimes_{\min} A$. 
    \end{proof}
  \begin{lem}\label{02}  Let $E\subset A$ be a f.d. subspace of a $C^*$-algebra.
  Let $q: C \to B$ be a surjective $*$-homomorphism.
  Let $\a$ be either $\min$ or $\max$. Assume
  $C \otimes_\a A/ \ker(q) \otimes_\a A \simeq B \otimes_\a A$ (which always holds for $\a=\max$).
  Then $q \otimes Id_E$ is a metric surjection from
  $(C\otimes E, \|\cdot \|_{\a})$  
  onto $(B\otimes E, \|\cdot \|_{\a})$.
  More precisely, this map takes the closed unit ball onto the closed unit ball.
    \end{lem}
   \begin{proof} The argument for this is a classical
   one based on the existence of nice approximate units in  ideals
   such as $\ker(q)$,  essentially  going back to Arveson.
   For the first assertion full details can be found in \cite[\S 4.7]{P6}. 
   For the case $\a=\max$, see also \cite[\S 7.2]{P6} on the ``projectivity" of the max-tensor product.
   The second assertion can be proved by the same principle as in
  Lemma A.32 in \cite{P6} or \cite[Lemma 2.4.7]{P4}.
  \end{proof}
  
  \begin{rem}\label{rr2'}  
   If a $C^*$-algebra $C$ has LLP and is QWEP
   then it is WEP.\\
   Indeed, let $q: W \to C$ be the quotient map from a WEP $C^*$-algebra $W$. Since $Id_C$ locally lifts up into $W$, the map
   $Id_C$ is $\C$-nuclear, and hence $C$ has the WEP.
     \end{rem}

Let $E\subset A$ be an operator space
  sitting in a $C^*$-algebra $A$.
  Let $D$ be another $C^*$-algebra. 
  Recall we denote (abusively) by $D  \otimes_{\max }  E$
  the closure of $D  \otimes   E$ in $D\otimes_{\max}  A$, and we denote
  by $\| \  \|_{\max}$ the norm induced
  on  $D  \otimes_{\max }  E$ by $D  \otimes_{\max }  A$.
  We define similarly $E  \otimes_{\max }  D$. 
  We should emphasize that
  $D  \otimes_{\max }  E$ (or $E  \otimes_{\max }  D$) depends on $A$ and on the embedding
  $E\subset A$, but there will be no risk of confusion.
  
  For any linear map $u: E \to C$
  We denote
  $$\|u\|_{ mb}= \{ \| 
  Id_{ \C}\otimes u: \C \otimes_{\max } E \to  \C \otimes_{\max } C\| \}$$
  
  Since any separable $C^*$-algebra is a quotient of $\C$ it is easy to see that
  $$\|u\|_{ mb}= \sup\{ \| 
 u_D: D \otimes_{\max } E \to  D \otimes_{\max } C\| \}$$
 where the sup runs over all possible $C^*$-algebras $D$'s.
  \\ In general \begin{equation}\label{eglo1}\|u\|_{ cb}\le \|u\|_{ mb} .\end{equation}
  However,   if $C$ is WEP,   we clearly have (since $\C\otimes_{\min} C= \C\otimes_{\max} C$)
    \begin{equation}\label{eglo}  \|u\|_{mb}= \|u\|_{cb} \end{equation}
     for any $u:E \to C$.

  We set
  $$MB(E,C)=\{u: E \to C\mid \|u\|_{mb} <\infty\},$$
  and we equip it with the mb-norm.
 
  We recall (see \cite[Th. 7.4 p. 137]{P6})
  that 
  $$\|u\|_{mb}=\inf\{ \|\tilde u: A \to C^{**} \|_{dec} \}$$
  where the inf runs over all $\tilde u: A \to C^{**}$  extending $i_C u: E \to C^{**}$.   \\   
  Here and throughout this note
  $$i_C:C  \to C^{**}$$ denotes the canonical inclusion.\\
  In particular if $C$ is a von Neumann algebra and $E=A$, we have 
  $\|u\|_{mb} = \|u\|_{dec} $.
  
  For any map $u: A \to C$ (defined on the whole of $A$) we have
  $\|i_C u\|_{mb}=\|i_C u\|_{dec}= \|u\|_{mb} \le \|u\|_{dec} $.
  Moreover, if $u$ is c.p. we have $ \|u\|_{dec} =\|u\|  $, and in the unital case this
  is $=\|u(1)\|  $ (see e.g.  \cite{P6}).

  \begin{pro}\label{p17} Assume $A$ separable with LP. 
  Let  $C/\cl I$ be a quotient $C^*$-algebra, let $q: C \to C/\cl I$ denote the quotient map.
  Then 
   any $u\in {MB}(A, C/\cl I)$ admits a lifting
$\hat u: A \to C$ with $\|\hat u\|_{{MB}(A,C)} =\|  u\|_{{MB}(A,C/\cl I)}.$
  \end{pro}
   \begin{proof}  See \cite[Th. 4.3]{157}.
  \end{proof}

  It is worthwhile to recall here that  a WEP $C^*$-algebra
  locally embeds in $\C$
  (in the sense of Definition \ref{d10})  if and only if  it has the LLP (see \cite[Prop. 3.7]{155}).
  
  We will use the following extension property for  
 maps from a subspace of an LLP to a WEP due (in some variant) to Kirchberg. See  \cite[Th. 21.4 and Remark 21.5 p. 360]{P6}
 for a detailed proof.
\begin{pro}\label{ext6} Let $X$ be a separable o.s. that locally embeds in $\C$ and let $A$ be a WEP $C^*$-algebra.
Then, for any  $\vp>0$,  any $u: E \to A$ defined on a
f.d. subspace $E\subset X$
admits an extension $\tilde u: X \to A$ with
$\|\tilde u\|_{cb}\le (1+\vp) \|u\|_{cb}$.
\end{pro}

  We will say that a 
  linear map $u: D \to C$ between $C^*$-algebras $D,C$
   locally  decomposably factors
  (resp. locally mb-factors) through
  another one $A$ if  there is a net of  maps
  $u_i: D \to A $, $v_i : A \to C$ with $\|u_i\|_{dec} \|v_i\|_{dec} \le 1$
  (resp.  $\|u_i\|_{mb} \|v_i\|_{mb} \le 1$)
  such that $v_iu_i: D \to C$ tends pointwise to $u$.
  
  \begin{pro}\label{pp1} Let $A$ be   WEP, LLP and such that 
  $\C$ locally embeds in   $A$. Then for any WEP and LLP 
  $C^*$-algebra $C$, the identity map $Id_C$  locally  mb-factors through
    $A$.
  \end{pro}
  \begin{proof}   
  Let $C$ be LLP. Then $C$ locally embeds in $\C$ and hence in $A$.
  Let $E\subset C$ be a f.d. subspace.
  Let $\vp>0$ and let $F\subset A$ such that   $d_{cb}(E,F) <1+\vp$. Let
 $u: E \to F$ such that   $\|u\|_{cb} <1$ and $\|u^{-1}\|_{cb} <1+\vp$.
   By Proposition \ref{ext6}
   and the fact that $A$ and $C$ are WEP (which guarantees
   $\|U\|_{mb}=\|U\|_{cb}$ for maps with range either $A$ or $C$)
   there are maps $U : C \to A$ and $V: A \to C$ respectively extending
    $u $ and   $u^{-1} $
   such that $\|U\|_{mb}<1$ and $\|V\|_{mb}<1+\vp$.
   Then $VU_{|E}= Id_E$.
  Taking for index set the set of pairs $(E,\vp)$
  and setting $u_i=U$ and $v_i= (1+\vp)^{-1} V$ for $i= (E,\vp)$, we obtain
  the announced result.
  \end{proof}

  \begin{rem}\label{rr2} Let $D$ be another $C^*$-algebra.
   If  $Id_C$ locally  decomposably factors (resp. locally mb-factors) through
   $A$, then
   $$A \otimes_{\min} D=A \otimes_{\max}  D \Rightarrow C \otimes_{\min}  D=C \otimes_{\max} D.$$
   This assertion is easy to check 
   using
   the fact that decomposable maps
   are ``tensorizable" for the maximal tensor product (see \cite[p. 138]{P6});
   or more precisely, that any map $u$ between $C^*$-algebras satisfies
   $$\|u\|_{cb}\le \|u\|_{mb} \le  \|u\|_{dec}  .$$
   Moreover, assuming $A$ and $C$ separable,  if $A$ has the LP then so does $C$. This can be checked using Theorem \ref{t15}.
  \end{rem}
 By Remark \ref{rr2} we deduce from the preceding proposition:
 
      \begin{cor}\label{c16} Let $A_1,A_2$ be  $C^*$-algebras with WEP and LLP.
      Assume that $\C$ locally embeds in each of them.
       Then for any $C^*$-algebra $D$
      $$A_1 \otimes_{\min} D=A_1 \otimes_{\max}  D \Leftrightarrow
      A_2 \otimes_{\min} D=A_2 \otimes_{\max}  D .$$
      Moreover, assuming both separable, $A_1$ has the LP
      if and only if so does $A_2$. 
    \end{cor}
    
   \begin{pro}\label{3} If there is an $A$ with WEP and LP such that
  $\C$  locally embeds in $A$, then for any QWEP $C^*$-algebra $C$    the implication LLP $\Rightarrow$ LP holds.
  \end{pro}
  \begin{proof} This follows immediately from Proposition \ref{pp1}
   and Remarks \ref{rr2} and \ref{rr2'}.
     \end{proof}

     \begin{lem}  Let $C,D$ be   $C^*$-algebras.
   Let $I$ be any index set.\\
  If $D$ is separable then for any 
  $s\in \ell_\infty(I;D) \otimes C$ we have
  \begin{equation}\label{ee1} \|s\|_{\max}=\sup \|(p_J \otimes Id_C) s\|_{ \ell_\infty(J;D) \otimes_{\max} C}\end{equation}
  where the sup runs over all countable subsets $J\subset I$ and $p_J: \ell_\infty(I;D)\to \ell_\infty(J;D) $ denotes the canonical    map of restriction to $J$.
  \end{lem}
  
   \begin{proof} 
  Fix $\vp>0$. Let $x\in \ell_\infty(I;D)$ viewed as a bounded function $x: I \to D$.
  Let $\{B_n\mid n\in \NN\}$ be a covering of 
    $D$ 
  by disjoint sets of diameter $<\vp$. 
  Note that the disjoint sets $I^x_n =x^{-1} (B_n)$  ($n\in \N$) cover $I$.
  We choose and  fix $i_n\in I^x_n$. Let $J=\{i_n\}$. Note
  that $J$ depends on $x$.
  Let $f_J: \ell_\infty(J;D) \to \ell_\infty(I;D)$ be the embedding defined by
  $$f_J(y)= \sum\nl_n 1_{I^x_n} y(i_n).$$
  Clearly 
  $  \| f_J p_J (x) - x\| \le \vp$.
  We now apply the same idea with $x$ replaced by a finite set
  $\{x\mid x\in S\}$.
  We replace   $\{I^x_n\mid n\in \N\}$ by the still countable partition
  $$ \{  \cap_{x\in S } I^x_{n(x)} \mid n  \in \N^S\}.$$
  With the obvious analogue of $f_J$ we now have 
  $$ \forall x\in S\quad  \| f_J p_J (x) - x\| \le \vp.$$
   Let $s\in \ell_\infty(I;D) \otimes C$.
   By what precedes we can find for any $\vp>0$ 
  a countable $J$ such that 
  $$\|( f_J p_J  \otimes Id_C)(s) -s \|_{\max} \le \vp.$$
  Therefore
  $$\|s \|_{\max} \le \|( f_J p_J  \otimes Id_C)(s)   \|_{\max} + \vp
  \le \|(  p_J  \otimes Id_C)(s)  \|_{\max} + \vp.$$
  This proves \eqref{ee1}.
     \end{proof}

     \begin{pro}\label{pp2}  Let $X$ be an   o.s. and  $C,D$  $C^*$-algebras.
  If    a linear map $u: X \to C$ is $\ell_\infty(D)$-nuclear with $D$ separable
  then it is $\ell_\infty(I;D)$-nuclear for any index set $I$. 
  \end{pro}
  
   \begin{proof}  
    This is immediate by \eqref{ee1}.
     \end{proof}
     \begin{rem}
     By Proposition \ref{pp2} we may replace ${\ell_\infty(\C)}$ by ${\ell_\infty(I;\C)}$
for any infinite set $I$ in \eqref{mM}.
\end{rem}

\section{Controllable operator spaces}

To abbreviate, we introduce
a new definition.
\begin{dfn} 
Let $\a$ be a $C^*$-norm on $\C \otimes A$ (e.g. $\a=\min$ or $\a=\max$).
 A f.d. subspace $E\subset A$ will be called $\a$-controllable 
if there is $t_E\in B_{\C \otimes_{\a} E} $ such that for any $t\in B_{\C \otimes_{\a} E} $ there is a unital $*$-homomorphism
 $\pi: \C \to \C$ such that $t=[\pi\otimes Id_E]( t_E)$.

\end{dfn}

     The next result was stated in \cite[Th. 10.5]{157}.
We include the proof for convenience.

 \begin{pro}\label{LP1} A    (resp. separable) $C^*$-algebra $A$  
 satisfies \eqref{e15} (resp. has the LP)    if and only if  
 any f.d. subspace $E\subset A$   is $\max$-controllable.  \end{pro}
  
   \begin{proof}  Assume $A$ satisfies  \eqref{e15}.
   Let $I= B_{\C \otimes_{\max} E} $.
   Let $X=(X_T)_{T\in I}\in \ell_\infty(I, \C \otimes_{\max} E)$ be defined by
   $X_T=T$.  Then clearly $\|X\|=1$. By \eqref{e15} 
      we have $\|X\|_{  \ell_\infty(I, \C) \otimes_{\max} E} =1$.
    Let $C=C^*(\F)$ with $\F$ a sufficiently big free group so that
   we have a quotient $*$-homomorphism
   $q: C \to \ell_\infty(I;\C)$. By Lemma \ref{02} there is $T_E\in B_{C \otimes_{\max} E} $
   lifting $X$. 
    Since a finite set of elements of $C=C^*(\F)$ 
   is ``supported" by a countable free subgroup, 
   there is a copy of $\C$, say $\C\simeq C' \subset C$
   such that  $T_E\in C' \otimes E$ and ${C' \otimes_{\max}  \C} \subset {C\otimes_{\max}  \C}$ isometrically (see Lemma \ref{nsf}).
Thus $T_E$ can be identified with an element $T_E\in B_{C' \otimes_{\max}  E}$.
   Let $\pi_i: \ell_\infty(I;\C) \to \C$ denote the $i$-th coordinate. Then it is easy to check
   that for any $T\in I$   we have  
   $T=[\pi\otimes Id_E]( T_E)$.
   with $\pi = \pi_T q_{|C'}$. Since $C'$ is a copy of $\C$, this proves the only if part.  \\
   To check the if part, assume $E$ $\max$-controllable and controled by
   $t_E\in B_{\C \otimes_{\max} E} $;
   if $t$ and $I$ are  as in Theorem \ref{t15}, with $t_i\in B_{\C \otimes_{\max} E} $
     then    
   for any $i\in I$ there is a unital $*$-homomorphism $\pi_i: \C \to \C$  
   such that $t_i= [\pi_i\otimes Id_E] (t_E)$ and hence $t= [\pi\otimes Id_E] (t_E)$
   where $\pi =(\pi_i)_{i\in I}: \C \to \ell_\infty(I,\C)$ is also 
   a  unital $*$-homomorphism. Then
   $\|t\|_{\ell_\infty(I,\C) \otimes_{\max} E} \le \|t_E\|_{\max}\le 1$.
    By homogeneity  \eqref{e15} follows. \end{proof}
     
We will need to consider the analogue of the property in Proposition \ref{LP1}
for the minimal tensor product. 
     
     \begin{dfn}
     A f.d. o.s. $E$ will be called controllable if it is 
$\min$-controllable.
 In that case we will say that $t_E$ controls $B_{\C \otimes_{\min} E} $
 or simply that it controls $E$.
\end{dfn}

By Proposition \ref{LP1},  if $A$ satisfies \eqref{e15} (e.g. is separable with LP) and has the WEP then any f.d. $E \subset A$ is controllable.
      
       \begin{pro}  \label{31}
     Let $E $ be a f.d.o.s. The following are equivalent:
     \begin{itemize}
     \item[{\rm (i)}] The space $E$ is  controllable.
     \item[{\rm(ii)}] For any set  $I$, any free group $\F$   and any 
     surjective unital $*$-homomorphism $q: C^*(\F) \to \ell_\infty(I,\C)$  
     the mapping $q\otimes Id_E$ is a metric surjection
     from ${C^*(\F) \otimes_{\min} E}$ to ${\ell_\infty(I,\C) \otimes_{\min} E}$.
     \item[{\rm(iii)}] Same as (ii) for $I=\N$ and some $q$.
     \end{itemize}
     \end{pro}
      \begin{proof} (i) $\Rightarrow$ (ii) : For any $i\in I$ 
      and $t_i\in B_{\C \otimes_{\min} E} $ there is a unital $*$-homomorphism $\pi_i: \C \to \C$ such that $t_i= (\pi_i \otimes Id_E)(t_E)$.
      Let $\pi=(\pi_i): \C \to \ell_\infty(I,\C) $.
      By the LP of $\C$ there is a unital c.p. map $\hat \pi: \C \to  C^*(\F)$
      lifting $q$. Let $\hat t= (\hat \pi\otimes Id_E)( t_E)$.
      Then $ (q\otimes Id_E) (\hat t) =(t_i)$. This shows that (ii) holds.\\  
      (ii) $\Rightarrow$ (i) : Let $I=B_{\C \otimes_{\min} E}$, and let $t\in \ell_\infty(I, \C\otimes E)\simeq \ell_\infty(I, \C)\otimes E$ be defined by $t_i=i$. Let $t_E\in B_{C^*(\F) \otimes_{\min} E}$ be a lifting of $t$.
      Let $q_i: C^*(\F) \to \C$ be the $i$-th coordinate of $q$.
       We have $ (q_i \otimes Id_E)(t_E)=t_i$. Using Lemma \ref{nsf} we may assume 
   $\F=\F_\infty$ and $t_E\in \C$ and we obtain (i).\\     
      (ii) $\Rightarrow$ (iii) is trivial.\\
      (iii) $\Rightarrow$ (ii) : Assume (iii). 
      By the same reasoning as for (ii) $\Rightarrow$ (i) we obtain 
      $t_E\in B_{\C \otimes_{\min} E}$ such that $\{(q_n \otimes Id_E)(t_E)\mid n\in \N\}$ is dense in $B_{\C \otimes_{\min} E}$. Let $I$ be an arbitrary set.
      Observe that
       by the separability of $\C$ any  finite set of elements of $\ell_\infty(I, \C)$
      can be approximated by ones that are constant on a common countable partition
      of $I$. With this observation, the argument for (i) $\Rightarrow$ (ii)  now shows that the image under $q\otimes Id_E$ of  
      $B_{\C \otimes_{\min} E}$
      is dense in
      the unit ball
      of $\ell_\infty(I, \C)$.
      By the open mapping theorem, this means that (ii) holds.
  \end{proof}
  
  By well known properties of ideals in $C^*$-algebras (see e.g. \cite[p. 103]{P6}) it can be shown
  that if $E$ satisfies  (ii) (or (iii)) in Proposition \ref{31} then any subspace of $E$ also does.
  This implies the next statement which can also be derived from Theorem \ref{t11} below.
  \begin{cor} Any subspace of a f.d. controllable $E $ is controllable.
\end{cor}
       \begin{rem}\label{r2}
       Let $q: C^*(\F) \to \ell_\infty(I,\C)$ be as in Proposition \ref{31}.
        Let $\cl I=\ker(q)$ so that $\ell_\infty(I,\C)=C^*(\F)/\cl I$.
        With the latter identification,
        we have a natural map
        $$\Phi_E: \ell_\infty(I,\C) \otimes_{\min} E  \to [C^*(\F)\otimes_{\min} E]/[\cl I\otimes_{\min} E],$$
        and $E$ is controllable if and only if
        $\| \Phi_E \|=1.$
        Assume that $E\subset B(H)$ is a unital subspace spanned by unitary operators. Let $A$ be the
        $C^*$-algebra generated by $E$.
        Note that $[C^*(\F)\otimes_{\min} E]/[\cl I\otimes_{\min} E]$
        naturally embeds in $[C^*(\F)\otimes_{\min} A]/[\cl I\otimes_{\min} A]$
        so that $\Phi_A$ can be identified with  the unital $*$-homomorphism extending $\Phi_E$.
        Then the linearization trick of \cite{93.}
        shows that $\| \Phi_E \|_{cb}=1$
        if and only if $\| \Phi_A \|=1$, or equivalently
        if and only if every f.d. subspace of $A$ is controllable.
        It is easy to check that $\| Id_{M_k} \otimes \Phi_E \|=1$
         if and only if $M_k(E)$ is controllable.
         Thus $M_k(E)$  is controllable for all $k\ge 1$
         if and only if  every f.d. subspace of $A$ is controllable.

       \end{rem}
     \begin{rem} Let $A$ be a $C^*$-algebra.
     The assertion that (ii) holds for any f.d. subspace $E\subset A$
     is equivalent to:
       \begin{itemize}
     \item[{\rm (ii)'}] The sequence 
     $$\{0\} \to  {\cl I \otimes_{\min} A} \to   {C^*(\F) \otimes_{\min} A}\to {\ell_\infty(I,\C) \otimes_{\min} A}$$
      \end{itemize} is exact
       where $\cl I=\ker(q) $.
       Therefore, any f.d. subspace of an exact  $C^*$-algebra is controllable.\\
      
     \end{rem}

      \begin{lem}\label{l11} 
Let $u: E \to B$ be a linear map from a f.d.o.s. to a $C^*$-algebra.\\
For any $C^*$-algebra $C$, we denote
$\|u\|_{C}= \|Id_C \otimes u : C \otimes_{\min} E \to C \otimes_{\max} B \|.$\\
If $E$ is controllable 
with $t_E$ controling $B_{\C \otimes_{\min} E} $
we have $\|u\|_{\C}     = \|[Id_\C \otimes u] (t_E)\|_{\max}$,\\ and 
$\|u\|_{\ell_\infty(\C) }= \|u\|_{\C}     $ 
or equivalently by \eqref{mM}
\begin{equation}\label{e29}
\|u\|_{mM} = \|u\|_{\C} =   \|[Id_\C \otimes u] (t_E)\|_{\max} .\end{equation}
If $E$ is controllable and $B$ WEP then $\|u\|_{mM} = \|u\|_{cb}     .$ 
  \end{lem}
  
   \begin{proof} Assume  that $t_E\in B_{\C \otimes_{\min} E}$   controls $E$. 
   For any $t\in B_{\C \otimes_{\min} E}$ there is a unital $*$-homomorphism
   $\pi_t: \C \to \C$ such that $t=[\pi_t \otimes Id_E] (t_E)$. We have
   $ [Id_\C \otimes u] (t)= [\pi_t \otimes u] (t_E)=[\pi_t \otimes Id_B] [ Id_\C \otimes u] (t_E)$ and hence
   $\| [Id_\C \otimes u] (t) \|_{\max} \le \| [Id_\C \otimes u] (t_E) \|_{\max}$, which proves the first equality.
   
   Let $I$ be any set, let $t$ be in the unit ball of
    $\ell_\infty(I,\C) \otimes_{\min} E$ or equivalently of $\ell_\infty(I,\C \otimes_{\min} E)$.
We have
$t=(t_i)_{i\in I}$ with  $t_i\in B_{\C \otimes_{\min} E}$.
For any $i\in I$ there is a unital $*$-homomorphism $\pi_i: \C \to \C$  
   such that $t_i= [\pi_i\otimes Id] (t_E)$ and hence $t= [\pi\otimes Id_E] (t_E)$
   where $\pi =(\pi_i)_{i\in I}: \C \to \ell_\infty(I,\C)$ is  
   a  unital $*$-homomorphism. Now $(Id_{\ell_\infty(I,\C)} \otimes u) ( t)=  [\pi\otimes u] (t_E)= [\pi\otimes Id_{B}][Id_\C \otimes u] (t_E)$
   and hence $\|(Id_{\ell_\infty(I,\C)} \otimes u) ( t) \|_{\ell_\infty(I,\C) \otimes_{\max} B} 
   \le \| [Id_\C \otimes u] (t_E)   \|_{\C \otimes_{\max} B} =    \|u\|_{\C}$.
   This shows that $\|u\|_{mM} \le  \|u\|_{\C}$. The converse is obvious.
   If $B$ has the WEP
then $\C \otimes_{\min} B=\C \otimes_{\max} B$   and hence $\|u\|_{\C}  = \|u\|_{cb}.   $ 
   \end{proof}

\begin{thm}\label{t11} Let $E\subset B(H)$ be a f.d.o.s. The following assertions are equivalent:
 \begin{itemize}  
 \item[{\rm (i)}] $E$ is controllable.
   \item[{\rm (ii)}] The   inclusion $i_E: E \to B(H)$ satisfies
$\|i_E\|_{mM}=1$.
 \item[{\rm (iii)}] There is a completely isometric
 map    $u$ from  $E$ to a $C^*$-algebra $B$ such that
 $\|u\|_{mM}\le 1 $.
 \item[{\rm (iv)}] There is a completely isometric
 map $\hat u$ from   $E$ to   $\C$ such that
 $\|\hat u\|_{mM}\le 1 $.
\item[{\rm (v)}] There is a c.i. embedding $f: E \to \C$ 
    such that the min and max norms coincide on $\C \otimes f(E)$.
\end{itemize} 
Moreover, if $E$ is controllable and $B$   WEP (e.g. $B=B(H)$) then any completely isometric
 $u:E \to B$ satisfies $\|u\|_{mM}= 1 $.
\end{thm}

\begin{proof} (i) $\Rightarrow$ (ii): Since $B(H)$ has the  WEP 
 this follows from Lemma \ref{l11}.\\
   (ii) $\Rightarrow$ (iii) is trivial.\\
 (iii) $\Rightarrow$ (iv): Assume (iii). Using an embedding $B\subset B(\cl H)$ we may assume
  $B= B(\cl H)$.
 Let $u:E \to B$ be  completely isometric.
 Let $\F$ be a big enough free group so that there is a quotient map $q: C^*(\F) \to B$.
    Let $v: u(E) \to C^*(\F)$ be a {\it linear} lifting.
  Let $T\in B_{\C \otimes_{\min} E} $ controling $B_{\C \otimes_{\min} E} $.
  Then ${\C \otimes_{\min} B} ={\C \otimes_{\max} B}$
  (which holds since $B= B(\cl H)$ has the WEP)  implies  $\|(Id_\C \otimes u)(T)\|_{\max} \le 1$.  Let $F\subset \C$ be f.d.
  such that $T\in F \otimes E$. By Lemma \ref{02}
  there is $\hat T \in B_{F \otimes_{\max} C^*(\F)} $ lifting $(Id \otimes u)(T)$.
  Let $0\le \sigma_i \le 1$ be an approximate unit in  $I=\ker(q)$.
  Let $w_i(x): =(1-\sigma_i) x$, $x\in C^*(\F)$. Then $(Id_\C \otimes vu) (T)-\hat T \in F \otimes I$ and hence
  $\lim_i \|(Id_\C \otimes w_i) [(Id_\C \otimes vu) (T)-\hat T]\|_{\max} =0$.
  Let $u_i=w_ivu: E \to C^*(\F)$. We have
  $\lim_i \|(Id_\C \otimes u_i)   (T) \|_{\max} \le 1$.
  Since $T$ controls $E$ it follows recalling \eqref{mM'} and \eqref{e29} 
  $$\lim\nl_i \|  u_i  \|_{mM} \le 1,$$
 and since $u_i$ lifts $u$, Arveson's principle gives us a lifting
   $\hat u: E \to C^*(\F)$ with $\|\hat u \|_{mM} = 1$. Since $u$ is c.i. so is $\hat u$.
   Lastly, being separable  the range of $\hat u$ is included in a copy $\C'$ of $\C$ sitting in
   $C^*(\F)$    and by Lemma \ref{nsf}  we obtain (iv).\\
(iv) $\Rightarrow$ (v):  Assume (iv).  A fortiori we have $\|\hat u\|_{\C}\le 1$.
This means that $f=\hat u$ satisfies (v).\\
(v) $\Rightarrow$ (i): Note  $f(E) \subset \C$ is $\max$-controllable by
Proposition \ref{LP1}. If (v) holds it must be $\min$-controllable, i.e.
controllable, and $f(E)$ is c.i. to $E$, so $E$ is controllable.\\
 When $B$ is WEP,  the last assertion follows from
 the last one in Lemma \ref{l11}.
\end{proof}
\begin{rem} An alternate proof of (iii) $\Rightarrow$ (iv) can be based on 
the local reflexivity in Theorem
\ref{t3}.
\end{rem}
\begin{cor}\label{c4} Any  controllable f.d.o.s. embeds completely isometrically in $\C$.
\end{cor}
\begin{rem}\label{jp10}
  By  \cite{JP}  
  there are  $3$-dimensional o.s. that 
  do {\it not}  embed c.i. in $\C$ and hence  are {\it not} controllable.
  \end{rem}
\begin{rem}\label{jp}
Let $E$ be a f.d.o.s. By  \cite{JP} (see also \cite[p. 351-352]{P4} or \cite[p. 351]{P6}) 
$E$ is c.i. to a subspace of $\C$ if and only if the same is true for $E^*$.
Incidentally, Corollary \ref{c4} can be deduced from Proposition \ref{31}
after observing that any f.d.o.s. $F$ embeds c.i. in $\ell_\infty(\C)$ and applying this to $F=E^*$.
The next statement shows that $E$ is controllable
        if and only if   its dual o.s. $E^*$ is c.i. to a rigid subspace of $\C$. \end{rem}
      \begin{lem}\label{LP4} 
  Let $E\subset \C$ be a  f.d. subspace.    Then $E$ is controllable if and only if 
   there is a completely isometric (c.i. in short) embedding $f: E^* \to \C$ on the dual o.s. such that
   any c.c. $v : f(E^*) \to \C$ is the restriction of a unital $*$-homomorphism $\pi : \C \to \C$.
\end{lem}
\begin{proof}  Assume $E$ controllable.
 We know by Remark \ref{jp} that there is a c.i. map $h: E^* \to \C$.
The latter corresponds to a tensor $t\in \C \otimes E$ with $\|t\|_{\min}=1$.
Since $E$ is controllable, there is $t_0 \in B_{\C \otimes_{\min} E} $
such that for any $s\in B_{\C \otimes_{\min} E} $ there is a unital $*$-homomorphism $\pi_s : \C \to \C$
such that $s=\pi_s \otimes Id_E (t_0)$. Let $f: E^* \to \C$  be associated to $t_0$.
In particular we have  $t=\pi_t \otimes Id_E (t_0)$, or equivalently $h= \pi_t f$.
This shows that $f: E^* \to \C$ is also a c.i. embedding. 
Now for any c.c. $v : f(E^*) \to \C$, $vf: E^* \to \C$ is c.c. Let $t_{vf}: E \otimes \C$ be the associated tensor so that
$ \|t_{vf}\|_{\min}=1$.
Using $s=t_{vf}$ we find $\pi_s$ such that 
$s=\pi_s \otimes Id_E (t_0)$, or equivalently $vf=    \pi_s f$ which means
$$v_{ |f(E^*)} = {\pi_s}_{ |f(E^*)}.$$
Conversely, given $f$,
 we set $F=f(E^*)^*$,
 and let $t_0\in \C \otimes F$
be associated to the inclusion $f(E^*) \to \C$. Let $t\in B_{\C \otimes_{\min} F} $, then $t$ is associated to a c.c. map $v: f(E^*) \to \C$, so it follows easily that $F$ is controllable. Since $F$  and $E$ are c.i. the same is true for $E$.
\end{proof}

        Let $E$ be a f.d. operator space. Recall  
  $$ex(E)=\inf\{ \|u\|_{cb} \|u^{-1} \|_{cb} \}$$
  where the inf runs over all $n\ge \dim(E)$ and all isomorphisms $u$
  from $E$ to a subspace of $M_n$.
  
  Let $C_u^*\langle X \rangle $ denote the unital universal $C^*$-algebra of an operator space $X$ with the canonical inclusion $k_X: X\to C_u^*\langle X \rangle $.  This unital $C^*$-algebra is characterized by the property
  that  any c.c. map $v: X \to D$ into an arbitrary unital $C^*$-algebra
  admits a unique extension to a unital $*$-homomorphism $\dot v: C_u^*\langle X \rangle \to D$. 
  The algebra $C_u^*\langle X \rangle$ is
  generated by $X$, i.e. it is  the smallest unital
  $C^*$-subalgebra containing $k_X(X)$.
  See e.g. \cite[p.160]{P4} or \cite[\S 2.7]{P6} for details.
  
   A separable operator space $X$ is said to have the OLP if
  any c.c. $u: X \to C/I$ into an arbitrary  quotient $C^*$-algebra
  admits a c.c. lifting.
  Ozawa proved in \cite{Ozllp} that this holds if and only if $C_u^*\langle X \rangle $
  has the LP.
  The next statement is but a simple reformulation of the OLP.
  We state this here to emphasize the comparison with the  subspaces that we call rigid.
    \begin{pro}\label{grr} 
    A separable operator space $X$ has the 
    OLP
 if and only if
  there is a c.i. embedding $f: X \to \C$ such that
any c.c. $v: f(X) \to D$ into an arbitrary unital $C^*$-algebra
admits an extension $\hat v: \C \to D$ that is a unital $*$-homomorphism.
  \end{pro}
    \begin{proof} 
    Assume $X$ has the OLP. By \cite{Ozllp} 
     this holds if and only if $C_u^*\langle X \rangle $ has the LP.
   We set $k=k_X: X \to C_u^*\langle X \rangle $ for simplicity.
  Let $q: \C \to C_u^*\langle X \rangle$ be a surjective unital $*$-homomorphism.
Then $k  $ admits  a c.c. lifting $f: X \to \C$, and the latter extends to a  unital
$*$-homomorphism $\sigma: C_u^*\langle X \rangle\to \C$. 
Note that $\sigma$ is a lifting of the identity on $C_u^*\langle X \rangle $,
and hence $\sigma$ defines a $*$-isomorphism between
$C_u^*\langle X \rangle $  and $\sigma(C_u^*\langle X \rangle )$. 
Let $\mu=\sigma^{-1}: \sigma(C_u^*\langle X \rangle ) \to C_u^*\langle X \rangle$.
Let $P=\sigma q: \C \to \C$.
Note that $P$ is a projection onto $\sigma(C_u^*\langle X \rangle)$ and $P$ is a unital $*$-homomorphism.
Let $v: f(X) \to D$ be a c.c. map.
Since $v f: X \to D$ is c.c. it admits an extension
$\pi: C_u^*\langle X \rangle \to D$ which is a unital $*$-homomorphism.
The composition $\hat v:=\pi   \mu P: \C \to D$ is a unital $*$-homomorphism that extends
$v: f(X)  \to D$.\\
Conversely, assume there is $f$ as in Proposition \ref{grr}.
We have a unital $*$-homomorphism
$\hat f: C_u^*\langle X \rangle \to \C$ extending $f$,
and also a unital $*$-homomorphism $\pi : \C \to C_u^*\langle X \rangle$ extending $f^{-1}_{|f(X)}$. Clearly the composition 
$\pi \hat f$ is the identity on $C_u^*\langle X \rangle$, which shows
  the latter factors through $\C$ via unital $*$-homomorphisms.
Since $\C$ has the LP,  $C_u^*\langle X \rangle$ also does.
Equivalently, $X$ has the OLP.
  \end{proof}
  \begin{cor} 
(i) Any f.d. $E$ with $ex(E)=1$ 
(e.g. $E=\ell_\infty^n$ or $E=M_n$) is controllable.
  \\ (ii) More generally, if $E$ locally embeds  in a WEP and LP $C^*$-algebra (e.g. in a nuclear algebra),  then   $E$ is controllable.
  \end{cor}
   \begin{proof} By
  \cite{Ozllp}  a f.d. operator space $E^*$ satisfies the OLP if and only if
  $ex(E)=1$. Thus (i) follows from  Proposition \ref{grr} and
  Lemma \ref{LP4}.
  \\
  Let $u: E \to B$ be an inclusion with $B$  having WEP and LP.
  By Proposition \ref{p1} we have $\|u\|_{mM} =1$ and hence
  (ii) follows from Theorem \ref{t11}.
\end{proof}

  \begin{rem} Let $n\ge 3$.
  We conjecture that $\ell_1^n$ is not controllable.
  Equivalently (see Lemma \ref{LP4}) we conjecture that $\ell_\infty^n$
  is not c.i. to a rigid subspace of $\C$.\\
  Note that  $\ell_\infty^n$ fails the OLP and hence
 fails the stronger rigidity property 
  in Proposition \ref{grr}. Our conjecture boils down to the failure of the latter for $D=\C$.
  \end{rem}
  Nevertheless, the next statement shows that  $\ell_\infty^n$ is  somewhat ``close" to  being c.i. to a rigid o.s.
  
 \begin{pro} There is a c.i. and c.p.  map $f: \ell_\infty^n \to \C$ such that
 for any unital $C^*$-algebra $D$ 
 any contractive  and  positive map $u: f(\ell_\infty^n) \to D$ extends
 to a unital $*$-homomorphism $\pi: \C \to D$.
 \end{pro}
  \begin{proof}
  Let $I=\{w: \ell_\infty^n \to D\mid \text{positive contraction into} \ D \text{ unital separable}  \}$. 
  Let $v: \ell_\infty^n \to \ell_\infty(\{D_w\}) $ be defined by $v(x)=\oplus_{w\in I}  w(x)$.
  Then $v$ is a positive isometry. Let $q: C^*(\F) \to \ell_\infty(\{D_w\}) $ be a surjective unital $*$-homomorphism, for some large enough free group $\F$.
  By a  result  due to Vesterstr\o m \cite{Ves} 
  there is a positive (still isometric)  lifting $\tilde v:  \ell_\infty^n \to C^*(\F)$.
   Note that $\tilde v  $ is c.i. and c.p. because on one hand for maps into $\ell_\infty^n$ the
  norm and the cb-norm are equal, and on the other hand this also holds for positive maps on $\ell_\infty^n$ 
  and the latter are c.p. 
     It follows that any positive contraction $w: \tilde v(\ell_\infty^n) \to D$ is the restriction of a unital
  $*$-homomorphism.
Indeed,  by the same argument
  as in Proposition \ref{LP1}
   $w\tilde v : \ell_\infty^n \to D_w$
  is of the form $\pi \tilde v$ for some unital $*$-homomorphism $\pi: C^*(\F) \to D_w$. Thus we may take $f=\tilde v$
  except that we must replace $ C^*(\F)$ by $\C$ using  Lemma \ref{nsf}.
     \end{proof}
  \begin{rem}
  We could replace positive by $n$-positive for any fixed $n\ge 1$.
  \end{rem}
  
  \begin{rem}
  If we drop positive
  and set $I=\{w: \ell_\infty^n \to D\mid \text{contraction into} \ D \text{ unital separable}  \}$,
   we obtain  an isometric $f: \ell_\infty^n \to \C$ such that
 for any unital separable $C^*$-algebra $D$ 
 any contractive    map $u: f(\ell_\infty^n) \to D$ extends
 to a unital $*$-homomorphism $\pi: \C \to D$.
But this is nothing but the embedding of
$X=\ell_\infty^n$ equipped with its maximal o.s. structure
(which has the OLP) 
provided by Proposition \ref{grr}.
  \end{rem}
        
      \section{Main point}
     
     In the next theorem 
     we relate the problem whether LLP $\Rightarrow$ LP
     for WEP $C^*$-algebras to the notion of controllable subspace.
     Note that   (i) $\Rightarrow$ (vii) has already been proved
     in Proposition \ref{3}.

  \begin{thm}\label{t1} The following assertions are equivalent:
  
  \item{\rm (i)} There is a separable $C^*$-algebra
  $A$ with WEP and LP such that 
  $\C$  locally embeds in $A$ (and hence $A$ is not nuclear).
     \item{\rm (ii)}
    For any f.d. subspace $E\subset \C$  and any $\vp>0$
    there is a  subspace $F\subset \C$ with $d_{cb}(E,F)<1+\vp$
    such that 
    the min and max norms coincide on $F \otimes \C$.
    In other words the inclusion $F\to \C$ is $\C$-nuclear.
      \item{\rm (iii)} For any f.d. subspace $E\subset \C$, any 
      complete contraction $v: E \to B$ with values in  a WEP $C^*$-algebra $B$
      is   $\ell_\infty(\C)$-nuclear.
      
          \item{\rm (iv)} Any complete contraction
          $u: \C \to \B$  is   $\ell_\infty(\C)$-nuclear.
          
      \item{\rm (v)} The mapping $j$ is $\ell_\infty(\C)$-nuclear.

         \item{\rm (vi)} Any f.d. subspace $E\subset \C$   
         is controllable.
          
          \item{\rm (vii)} For separable 
  WEP (or    QWEP) $C^*$-algebras)  the LLP implies the LP.
        
  \end{thm}
   \begin{proof} (i) $\Rightarrow$ (ii) :   After unitization, we may assume $A$ unital. Let $q: \C \to A$ be a surjective $*$-homomorphism. By the LP
   of $A$ we have a u.c.p. lifting $r: A \to \C$. Let $Y=r(A)$. Then  
    $Y\subset \C$ is completely isometric to $A$.
   Moreover,  we have a u.c.p. projection    $P:  \C \to Y$ given by $P=rq$. 
Therefore the min and max norms
      are identical on $Y\otimes \C$.
      Indeed,  this follows from the WEP of $A$ once we observe
      $$Y \otimes_{\max} \C \simeq A\otimes_{\max} \C .$$
     To check the latter note that the isomorphism $q_{|Y}:Y \to A$ with inverse $r: A\to Y$ 
      induces an isomorphism  
      $q\otimes Id_{\C} : Y \otimes_{\max} \C \to A\otimes_{\max} \C $
      with inverse 
       $r\otimes Id_{\C} : A \otimes_{\max} \C \to  Y\otimes_{\max} \C $.
       \\ Now for any $E\subset \C$ there is $F\subset Y$
   with $d_{cb}(E,F)<1+\vp$, whence $u: E \to F$ with
   $\|u\|_{cb} \|u^{-1}\|_{cb}<1+\vp$.
      Now   the min and max norms
      are identical on $F \otimes \C \subset Y\otimes \C$.
        This yields (ii).\\
   (ii) $\Rightarrow$ (iii) :  
   Let $u: E \to F$ be as before and $\vp>0$. 
   By  Lemma \ref{l1} the inclusion 
   $i_F: F\to \C$ is  $\ell_\infty(\C)$-nuclear. Then, 
   by Proposition \ref{ext6} and \eqref{eglo}, 
   since $B$ has the WEP,
   $vu^{-1}: F \to B$ extends to a map $g: \C \to B$ with
   $\|g\|_{mb} \le \|u^{-1}\|_{cb} (1+\vp)$ such that $g i_F=v u^{-1}$ . We have
   $v_{|E} = g  i_F u $.
   It follows that if $c=  \|u \|_{cb}\|u^{-1}\|_{cb}(1+\vp)$ the map
   $(1/c) v_{|E}$ is $\ell_\infty(\C)$-nuclear.
   Since this holds for any $c>1$ 
   we obtain (iii).
   \\
    (iii) $\Rightarrow$ (iv) $\Rightarrow$ (v) : obvious.  \\
       Theorem \ref{t11} shows that (v)  $\Rightarrow$ (vi).\\
     (vi) $\Rightarrow$ (vii)  : 
         Assume (vi). Let $A$ be QWEP and LLP.
         By Remark \ref{rr2'} $A$ is automatically WEP and hence $\C \otimes_{\min} E=
         \C \otimes_{\max} E$ for any $E\subset A$.
         Then  $A$ has the LP by the criterion in Proposition \ref{LP1}.
         
   (vii) $\Rightarrow$ (i)  :  obvious by \cite{155}. Indeed, it is proved there 
   that there is an $A$ with WEP and LLP   such that $\C$ locally embeds in $A$. 
    \end{proof}
  
  \begin{rem} By Lemma \ref{LP4} the properties in Theorem \ref{t1} are also equivalent
  to\\
  \item{\rm (vi)*} Any f.d. subspace $E\subset \C$  is c.i. to a rigid subspace of $\C$.
  \end{rem}
  
  \begin{pro}\label{p16}  Let $A$ be  any separable WEP, LLP
  $C^*$-algebra such that  $\C$ locally embeds in $A$. Let 
  $X$ be an o.s. that locally embeds in $\C$ and let
  $j_X: X \to B(H)$ be a c.i. embedding.
   There is a net of c.c. maps $v_i: X \to A$, $w_i: A\to B(H)$
   such that the composition $w_iv_i : X \to B(H)$ tends pointwise
   to $j_X$.
  \end{pro}
   \begin{proof}
   Let $E\subset X$ be a f.d. subspace and let $\vp>0$.
    We have an  embedding $v: E\to A$ with $\|v\|_{cb} \le 1$
    and $\|{v^{-1}}_{|v(E)}\|_{cb} \le 1+\vp$. By Proposition \ref{ext6}
    the map $v$ 
    admits an extension $\tilde v : X \to A$ with
    $\|\tilde v\|_{cb} \le 1+\vp$. By the injectivity of $B(H)$ 
    there is a map $w: A \to B(H)$ extending $ {j_X }_{|E}{v^{-1}}_{|v(E)}: v(E) \to B(H)$ with $ \|w\|_{cb} \le 1+\vp$. 
Clearly ${w \tilde v}_{|E}= {j_X}_{|E}$. 
Thus we may index the desired net by pairs $(E,\vp)$ and  for
$i=(E,\vp)$   setting $v_i=(1+\vp)^{-1} \tilde v$ and $w_i=(1+\vp)^{-1} w$
gives the announced net.
   \end{proof}
  
 \begin{pro}\label{iglo} Assume that the equivalent properties in Theorem \ref{t1} hold. Let $X$ be a 
  separable o.s. that locally embeds in $\C$.
  Then there is a c.i. embedding $f: X \to \C$
 such that the min and max norms agree on $f(X)\otimes \C$.
  \end{pro}
  
    \begin{proof} Let $j_X: X \to B(H)$ be a c.i. embedding.
    Let $q: C^*(\F) \to B(H)$ be a surjective $*$-homomorphism.
     Our goal is to show that $j_X: X \to B(H)$ admits a lifting $f: X \to C^*(\F)$ such that $\|f\|_{mM}=1$.
    Since $j_X$ is completely isometric, so will be $f$. By  Arveson's principle
    it suffices to show the following claim:  there is a net of maps $f_i: X \to C^*(\F)$
    with $\lim_i \|f_i\|_{mM}\le 1$ such that $qf_i$  tends pointwise to $j_X$.
         Let $v_i,w_i$ be as in Proposition \ref{p16}.
    Note $\|w_i\|_{mb}=\|w_i\|_{cb}$ by \eqref{eglo}. By the LP of $A$ (see Proposition \ref{p17} and \eqref{eglo1})
    the map $w_i$ admits a lifting $\hat w_i: A \to C^*(\F)$ with $\|\hat w_i\|_{mb}\le 1 $. Moreover,  $\|\hat w_i\|_{mM}=\|\hat w_i\|_{mb} \le 1 $ by \eqref{mM11}.
    Let $f_i=  \hat w_i v_i : X \to C^*(\F)$. Note $qf_i =w_iv_i$
       and $\| f_i  \|_{mM} \le \|\hat w_i\|_{mM} \| v_i\|_{cb}\le 1$. 
    Thus the maps $(f_i)$ form a net  
    such that $q f_i \to j_X$ pointwise and $\lim \|f_i\|_{mM} \le 1$.
This proves our claim.
    By Arveson's principle, there is a c.c. map $f: X \to C^*(\F)$  lifting $j_X$.
    Then  $f$ is a c.i. embedding and      
    as before (by Lemma \ref{nsf})
    $\F$ can be replaced by $\F_\infty$.
   \end{proof}
        \begin{cor} If  the equivalent properties in Theorem \ref{t1} hold
        then any separable o.s. that locally embeds in $\C$ actually
        embeds completely isometrically in $\C$.
        \end{cor}
        
         The next statement shows that the equivalent assertions in Theorem \ref{t1} are actually
  equivalent to a ``global" version of (vi).
  
  \begin{cor}\label{kglo} For convenience let $\C_1$ be a copy of $\C$.
 The equivalent properties in Theorem \ref{t1} hold if and only if
  there is a c.i. embedding $f:\C \to \C_1$
  such that the min and max norms (induced by 
  the respective ones on $ \C_1\otimes \C$)
   coincide on $f(\C)\otimes \C \subset \C_1 \otimes \C$.
  \end{cor}
        
        \begin{proof}
        Assume that $f(\C)\otimes_{\min } \C =f(\C)\otimes_{\max } \C$.
        By (i) $\Leftrightarrow$ (v) in Theorem \ref{t11},
        any f.d. $E\subset \C$ is controllable, which is (vi) in
        Theorem \ref{t1}. This proves the if part. Conversely, by Proposition \ref{iglo}, the assertion (vi) in
        Theorem \ref{t1} implies the existence of $f$ as in Corollary \ref{kglo}.
        \end{proof}
        
        For convenience, 
        let us say that a $C^*$-algebra $D$ is $j$-nuclear
        if $j$ is $D$-nuclear.
  \begin{thm} 
   Let $A$ be  any separable WEP, LLP
  $C^*$-algebra such that  $\C$ locally embeds in $A$.
  Let $(D_\a)_{\a\in I}$ be an arbitrary family of 
  $j$-nuclear separable $C^*$-algebras. If $A$ has the LP
then $(A, \ell_\infty(\{D_\a\}) )$ is a nuclear pair 
and     $j$ is $ \ell_\infty(\{D_\a\})$-nuclear. Equivalently
$ \ell_\infty(\{D_\a\})$ is $ j$-nuclear.
  \end{thm}
         \begin{proof}  Fix $\a\in I$. Assume   that $D_\a$ is separable
         and $j$-nuclear. By the construction in \cite{155} there is a
          WEP and LLP separable $C^*$-algebra $A_\a$ such that  $\C$  locally embeds
          in $A_\a$ and is such that the pair $(A_\a,D_\a)$ is nuclear.
          By Corollary \ref{c16} the pairs $(A,D_\a)$ are nuclear
         for any $\a \in I$.  
         By Theorem \ref{t15} the pair
         $(A, \ell_\infty(\{D_\a\}) )$ is nuclear.
         We now apply Proposition \ref{p16} to $X=\C$ with $j$ in place of $j_X$.
         Recalling that $\|w_i\|_{cb}=\|w_i\|_{mb}$ (see \eqref{eglo}) we observe
         that the map $w_iv_i: \C \to \C$ is $\ell_\infty(\{D_a\})$ nuclear
         (since $Id_A$ is so). Taking the limit it follows that $j$ itself
         is $\ell_\infty(\{D_\a\})$ nuclear.            \end{proof}
  
  \begin{rem} 
  Thus if there is an $A$ with WEP and  LP 
    such that  $\C$ locally embeds in $A$,
  the class of $j$-nuclear $C^*$-algebras
  is stable by arbitrary $\ell_\infty$-sums.
      \end{rem}
      
   \section{A ``local reflexivity" theorem}

In Banach space theory, the local reflexivity (LR) principle from \cite{LR} plays 
an important role. For $C^*$-algebras, following seminal work
by Archbold and Batty, Effros and Haagerup studied in \cite{EH}
the cb-version of that principle. In \cite{157} we proved an
mb-version of the same principle. What follows is a more general  version of   the one in \cite{157}, suggested by our use of the norm $mM$ in the preceding section.

Let $E\subset A$ be a finite dimensional subspace of a  $C^*$-algebra
  $ A$.

 \begin{thm}\label{t3}  
Let $E\subset A$ be any f.d. subspace of a $C^*$-algebra $A$.
Then  for any separable $C^*$-algebra $C$ 
 we have a contractive inclusion
\begin{equation}\label{mblr}
{mM}(E,C^{**})\to  {mM}(E,C)^{**}.\end{equation}
In other words any $u$ in the unit ball of ${mM}(E,C^{**})$ is the weak* limit
of a net $(u_i)$ in the unit ball of ${mM}(E,C)$.
\end{thm}

More generally, consider a class of   $C^*$-algebras $\cl D$
that is stable by arbitrary $\ell_\infty$-sum, i.e. for any family
$(D_i)_{i\in I}$ in $\cl D$ we have $\ell_\infty(\{D_i\}) \in \cl D$.

Then we define
$$\|u\|_{mM^\cl D}= \sup_{D\in \cl D}\|id_D \otimes u: D \otimes_{\min} E \to D \otimes_{\max} C\|,$$
and we denote by ${mM_\cl D}(E,C)$ the space of all $u:E \to C$
equipped with this norm.

When $\cl D=\{\ell_\infty(I,\C)\mid I \ arbitrary\}$ and $C$ has LP then
$$\|u\|_{mM^\cl D}= \|u\|_{mM}.$$

\begin{thm}\label{t3'}  
Assume $\cl D$ stable by arbitrary $\ell_\infty$-sums.
Let $E\subset A$ be any f.d. subspace of a $C^*$-algebra $A$.
Then for any  $C^*$-algebra $C$
we have a contractive inclusion
\begin{equation}\label{mblr'}
{mM^\cl D}(E,C^{**})\to  {mM^\cl D}(E,C)^{**}.\end{equation}
 
\end{thm}

   \begin{proof}
  This will follow from the bipolar theorem.
  We first need to identify the dual of
   ${mM^\cl D}(E,C)$. As a vector space 
   ${mM^\cl D}(E,C) \simeq C  \otimes E^*$ and
   hence ${mM^\cl D}(E,C)^* \simeq C^* \otimes E$
(or say $({C^*})^{\dim(E)}$).
  We equip $C^* \otimes E$ with the norm $\a$ defined as follows.
  We call ``admissible" any $C^*$-algebra $D$ in $\cl D$.
  Let $K\subset {mM^\cl D}(E,C)^*$
   denote the set of those $f \in {mM^\cl D}(E,C)^*$ for which 
  there is an admissible $D$,    a functional  $w$ in the unit ball of
   $ {(D\otimes_{\max } C)^* } $
  and  
  $t \in B_{ D\otimes_{\min } E}$, 
  so that
  $$\forall u\in {mM^\cl D}(E,C)\quad f(u)= \langle w, [Id_D \otimes u](t)\rangle .$$
    By the very definition of $\|u\|_{{mM^\cl D}(E,C)}$ we have
   \begin{equation}\label{e3}\|u\|_{mM^ \cl D}= \sup\{ |f(u) | \mid f\in K\}.\end{equation}
  This implies that the   unit ball of the dual of ${mM^\cl D}(E,C)$ is the bipolar of $K$.
We will show that  actually $K$  is  the unit ball of
${mM^\cl D}(E,C)^*$.\\
First we show that $K$ is convex.
Let $D_1,D_2$ be admissible $C^*$-algebras. Let $D=D_1\oplus D_2$
with the usual $C^*$-norm. Note that   $D$ is admissible.
Using the easily checked identities
 (here the direct sum is in the $\ell_\infty$-sense) $$D \otimes_{\min} E= (D_1 \otimes_{\min} E) \oplus (D_2 \otimes_{\min} E), \text{  and  } {D\otimes_{\max } C  }=(D_1\otimes_{\max } C ) \oplus
(D_2\otimes_{\max } C) ,$$
and hence $ {(D\otimes_{\max } C)^* }=(D_1\otimes_{\max } C)^* \oplus_1
(D_2\otimes_{\max } C)^*$
(direct sum in the $\ell_1$-sense), it is easy to check that 
$K$ is convex and hence that $K$
is the unit ball of some norm $\a$ on   ${mM^\cl D}(E,C)^*$.\\
  Our main point  is the  claim that $K$ is weak* closed.
  To prove this, let $(f_i)$ be a net in $K$ converging
  weak* to some $f \in {mM^\cl D}(E,C)^*$.
  Let $D_i$ be admissible $C^*$-algebras, $w_i \in B_{(D_i\otimes_{\max} C)^*}$
  and $t_i \in B_{D_i \otimes_{\min} E}$ such that    we have 
  $$\forall u\in {mM^\cl D}(E,C)\quad f_i(u)= \langle w_i, [Id_{D_i} \otimes u](t_i)\rangle .$$
    Let $D= \ell_\infty(\{D_i\})$ and 
  let $t\in D\otimes E$ be associated to $(t_i)$.
  Clearly  $\|t\|_{\min}\le 1$.
  Let $p_i: D \to D_i$ denote the canonical coordinate projection,
  and let $ v_i\in (D\otimes_{\max} C)^*$ be the functional
  defined by $v_i(x)= w_i(  [p_i\otimes Id_C](x) )$.
  Clearly $v_i \in B_{(D\otimes_{\max} C)^*}$  and
 $  f_i(u)= \langle v_i, [Id_D \otimes u](t)\rangle .$
 Let $w$ be a weak* limit point of $(v_i)$ (or the limit point if we pass to an
 ultrafilter refining our net).
  By   weak* compactness,  $w\in B_{(D\otimes_{\max} C)^*}$.
  Then
  $f(u)=\lim f_i(u)= \langle w, [Id_D \otimes u](t)\rangle .$
  Thus we conclude $f\in K$, which proves our claim.
  
  By  \eqref{e3} the   unit ball of the dual of ${mM^\cl D}(E,C)$ is the bipolar of $K$,
  which is equal to its weak* closure. By what precedes,
  the latter coincides with $K$. Thus the gauge of $K$ is the 
  announced dual norm $\a=\|\  \|^*_{{mM^\cl D}}$.
  
  Let $u''\in {{mM^\cl D}(E,C^{**})}$ with $\|u''\|_{mM^\cl D}\le 1$. By the bipolar theorem,
  to complete the proof  it suffices to show that
  $u''$ belongs to the polar of $K$, or equivalently that
  $ |f(u'')|\le 1  $ for any $f\in K$.
  To show this consider  $f\in K$ taking 
  $  u\in mM^\cl D (E,C)$ to $ f(u)= \langle w, [Id_D \otimes u](t)\rangle $
  with $D$ admissible, $w \in B_{(D\otimes_{\max} C)^*}$
  and $t \in B_{D \otimes_{\min} E}$.
  \\
  Observe that  
  $[Id_D \otimes u''] (t) \in D \otimes C^{**} \subset (D \otimes_{\max} C)^{**}$.
  Recall that $mM^\cl D(E,C^{**})\simeq mM^\cl D(E,C)^{**} \simeq (C^{**})^{\dim(E)}$
  as vector spaces. Thus we may view  $f\in mM^\cl D(E,C)^{*}$ 
  as a weak* continuous functional on $mM^\cl D(E,C^{**})$ to define 
$f(u'')$.
   We claim that 
  \begin{equation}\label{cb1}  f(u'')= \langle w, [Id_D \otimes u''](t)\rangle ,\end{equation}
  where the duality is relative  to
  the pair $\langle (D \otimes_{\max} C)^{*},(D \otimes_{\max} C)^{**} \rangle$.
  From this claim the conclusion is immediate. Indeed,
  we have 
 $\|[Id_D \otimes u''](t) \|_{D \otimes_{\max} C^{**}}\le \|u''\|_{mM^\cl D}\le 1$, and
by the maximality of the max-norm on $D  \otimes  C^{**} $
we have a fortiori 
$$\|[Id_D \otimes u''](t) \|_{(D \otimes_{\max} C)^{**} }\le \|[Id_D \otimes u''](t) \|_{D \otimes_{\max} C^{**}} \le 1.$$
Therefore  
$|f(u'')|=|\langle w, [Id_D \otimes u''](t)\rangle|  \le \|w\|_{(D \otimes_{\max} C)^{*}} \le 1$, which completes the proof modulo our claim \eqref{cb1}.\\
To verify the claim, note that 
the identity \eqref{cb1}  holds for any $u''\in mM^\cl D(E,C)$.
Thus it suffices to prove that the right-hand side of \eqref{cb1} 
is a  weak* continuous function of $u''$ (which is obvious for the left-hand side). 
To check this
one way is to note that $t\in D \otimes E$ can be written
as a finite sum $t=\sum d_k \otimes e_k$ ($d_k\in D,e_k\in E$) and 
if we denote by $\dot w: D \to C^*$ the linear map associated to $w$
we have
$$\langle w, [Id_D \otimes u''](t)\rangle
=\sum\nl_k \langle w, [d_k \otimes u''(e_k)]\rangle
=\sum\nl_k  \langle  \dot w(d_k) , u''(e_k) \rangle, $$
and since  $\dot w(d_k) \in C^*$ the weak* continuity as a function of $u''$ is obvious,
completing the proof.
  \end{proof}

\begin{rem}\label{fg} Let $E\subset A$ be f.d.
and
 let $\cl D$ be   as before stable by arbitrary $\ell_\infty$-sums. For the sake of the comparison we wish to make below, we define for any $u: E \to C$, 
\def\g{\gamma}
$$\g_{1}(u)= \sup_{D\in \cl D}\|Id_{D} \otimes u: {D} \otimes_{\min} E
\to {D} \otimes_{\max} C\|.$$
For instance, if
 $\cl D$ is formed of WEP $C^*$-algebras and if $ \{M_n\mid n\ge 1\}$ locally embeds in $\cl D$, then
 $$\g_{1}(u) =\|Id_\B \otimes u: \B \otimes_{\min} E \to \B\otimes_{\max} C\| .$$ In that case $\g_{1}$ is related to the LLP.
We   also wish to consider
$$\g_{2}(u)= \sup_{D\in \cl D}\|Id_{D} \otimes u: {D} \otimes_{\max} E
\to {D} \otimes_{\max} C\|.$$

The preceding proof (resp. the similar one in \cite{157})   shows that, without any additional assumption on $A,D,C$, for $\gamma=\gamma_1$ (resp. for $\gamma=\gamma_2$)  
 we have a contractive inclusion
  \begin{equation}\label{gty}{\Gamma}(E,C^{**})\to {\Gamma}(E,C)^{**},\end{equation}
where $ {\Gamma}(E,C)$ denotes the   space of 
all maps $u: E \to C$ equipped with the norm $\gamma$.

Assume $A$ has the LP or satisfies \eqref{e15}. Then when either $\cl D=\{\C\}$ or
 $\cl D=\{\ell_\infty(I,\C)\mid I \ arbitrary\}$ 
(or   when $\cl D$ is the class of all $C^*$-algebras) we have
$\gamma_2(u)=\|u\|_{mb}$. In that case \eqref{gty} is nothing but 
  the LR principle from \cite[Th. 5.2]{157}, which is equivalent to the LP of $A$.

 If $\cl D$ is the collection formed of all $\ell_\infty(I)$'s (i.e. we replace $\C$ by $\CC$ !)
then $\g_{1}(u)=\g_{2}(u)=\|u\|$ and we recover the classical  LR principle for Banach spaces (see \cite{LR}).

If $C$ (resp. $A$) is nuclear and if $\{M_n\}$ locally embeds in $\cl D$, then 
$\g_{1}(u)  =\|u\|_{cb}$  (resp. $\g_{2}(u) =\g_{1}(u)  $ for all $\cl D$)
 which is consistent with the fact that
nuclearity implies   LR (resp. LP).
\end{rem}

\begin{cor}\label{L21}  Let $\cl D$ be as in Theorem \ref{t3'}.
Let $C/\cl I$ be a quotient $C^*$-algebra with quotient map $q: C \to C/\cl I$.
Let $E \subset A$  a   f.d. subspace.  
Then 
any $u: E \to C/\cl I$
admits a lifting $\hat u: E \to C$ such that
$$\|\hat u \|_{mM^{\cl D}} = \|  u \|_{mM^{\cl D}} .$$
\end{cor}
\begin{proof} 
 Replacing $C/\cl I$ by $(C/\cl I)^{**}$ 
 and using the morphism lifting $(C/\cl I)^{**}$ up to $C^{**}$
 we can lift $u$
 to a map $v: E \to C^{**}$ with $\|v\|_{mM^{\cl D}}= \|  u \|_{mM^{\cl D}}$.
 Then by   Theorem \ref{t3'}
 there is a net $v_i: E \to C$
 with $\|v_i\|_{mM^{\cl D}} \le 1$ tending pointwise weak* to $v$.
 By the usual Mazur argument passing to suitable convex hulls
 we get a net such that $qv_i$ tends pointwise (in norm) to $u$.
 For the final point we use Arveson's principle, observing that
  the classes $$\cl F(E,C)=\{u:E \to C\mid \|u \|_{mM^{\cl D}}\le 1\}$$
 are admissible.
 \end{proof}

 \n\textbf{Acknowledgement.} I am grateful to Jean Roydor
    for stimulating conversations.

\end{document}